\documentclass[12pt]{amsart}
\usepackage{amssymb,amsmath,amsfonts,latexsym}
\usepackage{xcolor}

\newcommand{\newsection}[1]{\setcounter{equation}{0} \section{#1}}
\newtheorem{thm}{Theorem}[section]
\newtheorem{prop}[thm]{Proposition}
\newtheorem{lem}[thm]{Lemma}
\newtheorem{cor}[thm]{Corollary}

\newtheorem*{theorem*}{Theorem}
\newtheorem*{statement*}{Statement}

\theoremstyle{definition}
\newtheorem{definition}[thm]{Definition}

\theoremstyle{remark}

\numberwithin{equation}{section}
\newcommand{\D}{\mathbb{D}}

\newcommand{\C}{\mathbb{C}}

\newcommand{\Z}{\mathbb{Z}}

\newcommand{\clb}{\mathcal{B}}

\newcommand{\cle}{\mathcal{E}}

\newcommand{\clh}{\mathcal{H}}
\newcommand{\clk}{\mathcal{K}}

\newcommand{\cln}{\mathcal{N}}

\newcommand{\cls}{\mathcal{S}}

\newcommand{\clu}{\mathcal{U}}
\newcommand{\clw}{\mathcal{W}}

\newcommand{\raro}{\rightarrow}

\begin{document}

\title[Orthogonal decompositions II]{Orthogonal decompositions and twisted isometries II}


\author[Rakshit]{Narayan Rakshit}
\address{Narayan Rakshit, Department of Mathematics, Visvesvaraya National Institute of Technology, South Ambazari Road, Ambazari, Nagpur, Maharashtra 440010, India}
\email{narayan753@gmail.com}

\author[Sarkar]{Jaydeb Sarkar}
\address{J. Sarkar, Indian Statistical Institute, Statistics and Mathematics Unit, 8th Mile, Mysore Road, Bangalore, 560059,
India}
\email{jay@isibang.ac.in, jaydeb@gmail.com}

\author[Suryawanshi]{Mansi Suryawanshi}
\address{Mansi Suryawanshi, Statistics and Mathematics Unit, Indian Statistical Institute, 8th Mile, Mysore Road, Bangalore, Karnataka - 560059, India}
\email{mansisuryawanshi1@gmail.com}



\begin{abstract}
We classify tuples of (not necessarily commuting) isometries that admit von Neumann-Wold decomposition. We introduce the notion of twisted isometries for tuples of isometries and prove the existence of orthogonal decomposition for such tuples. The former classification is partially inspired by a result that was observed more than three decades ago by Gaspar and Suciu. And the latter result generalizes Popovici's orthogonal decompositions for pairs of commuting isometries to general tuples of twisted isometries which also includes the case of tuples of commuting isometries. Our results unify all the known orthogonal decomposition related results in the literature.
\end{abstract}


\maketitle


\section{Introduction}\label{sec:intro}

This article is concerned with orthogonal decompositions and representations of $n$-tuples of (not necessarily commuting) isometries acting on Hilbert spaces. We assume throughout that $n (>1)$ is a natural number and Hilbert spaces are separable and over $\C$. The present article can be considered as a sequel to \cite{RSS}, although fairly different, more general in spirit, and almost an independent read.

The starting point for our analysis is the classical von Neumann–Wold decompositions of isometries on Hilbert spaces. Let $\clh$ be an arbitrary but fixed Hilbert space. A bounded linear operator $V$ on $\clh$ (in short $V \in \clb(\clh)$) is said to be an isometry if $\|Vf\| = \|f\|$ for all $f \in \clh$, or equivalently, $V^* V = I_{\clh}$. Of course, a trivial example of an isometry is unitary, whereas a less trivial example is a shift operator. Recall that an isometry $V$ is a \textit{shift} if
\[
\text{SOT}-\lim_{m \rightarrow \infty} V^{*m} = 0.
\]
The classical von Neumann–Wold decomposition theorem says that
these examples of isometries are rather typical:

\begin{thm}[von Neumann–Wold decomposition]\label{thm: Wold}
Let $V \in \clb(\clh)$ be an isometry. Then $\clh_{\emptyset} := \cap_{m=0}^\infty V^m \clh$ reduces $V$. Moreover,  $V|_{\clh_{\emptyset}}$ is a unitary and $V|_{\clh_{\{1\}}}$ is a shift, where
\[
\clh_{\{1\}}:= \clh_{\emptyset}^\perp = \bigoplus_{m=0}^\infty V^m (\ker V^*).
\]
\end{thm}

Therefore, $V$ gives rise to an orthogonal direct sum of two closed subspaces (one of them is possibly zero) with respect to which $V$ admits a $2 \times 2$ diagonal block matrix whose diagonal blocks are shift and unitary operators. This simple description of isometries is one of the most fundamental results in linear analysis. For instance, Theorem \ref{thm: Wold} plays an essential role in prediction theory \cite{Helson}, time series analysis \cite{Masani}, stochastic process \cite{Miamee}, operator models \cite{Nagy}, $C^*$-algebras \cite{Coburn, Tron, W}, etc.

Theorem \ref{thm: Wold} raises natural and important questions, such as how to determine whether an $n$-tuple of isometries admits (meaningful) orthogonal direct sum decomposition of closed subspaces, and even if such decomposition exists, what are the representations of tuples of isometries, whether such representations are unique or canonical, and of course, what are the invariants. The general problem seems completely inaccessible as even the structure of commuting pairs of isometries is notoriously complicated and largely unknown. However, this problem has been settled for certain classes of isometries. Notably, in \cite{Slo}, S{\l}oci\'{n}ski demonstrated a completely satisfactory theory for pairs of doubly commuting isometries. In order to be more precise, let us clarify the meaning of orthogonal decompositions of tuples of isometries. The motivation comes from von Neumann \cite{von}, Wold \cite{Wold}, S{\l}oci\'{n}ski \cite{Slo}, and Popovici \cite{Pop}.

\begin{definition}\label{def: orth decomp}
An $n$-tuple of isometries $V = (V_1, \ldots, V_n)$ acting on $\clh$ admits an orthogonal decomposition if there exist $2^n$ closed subspaces $\{\clh_A\}_{A \subseteq I_n}$ of $\clh$ (some of these subspaces may be trivial) such that

\begin{enumerate}
\item $\clh_A$ reduces $V$ for all $A \subseteq I_n$,
\item $\clh = \bigoplus_{A \subseteq \{1,\ldots, n\}} \clh_A$, and

\item ${V_i}|_{\clh_A}$, $i \in A$, is a shift, and ${V_j}|_{\clh_A}$, $j \in A^c$, is a unitary for all $A \subsetneqq \{1,\ldots, n\}$.
\end{enumerate}

\noindent If, in addition, (3) holds for $A = \{1, \ldots, n\}$, then we say that $V$ admits a von Neumann–Wold decomposition.
\end{definition}

The meaning of condition (1) is that $\clh_A$ reduces $V_i$ for all $i=1, \ldots, n$. At this point, we pause to warn the reader that, in the definition of orthogonal decomposition; we do not impose any particular condition on the reducing subspace $\clh_{\{1, \ldots, n\}}$.

S{\l}oci\'{n}ski proved that pairs of doubly commuting isometries admit von Neumann–Wold decomposition. Recall that an $n$-tuple of isometries $(V_1, \ldots, V_n)$ on $\clh$ is said to be \textit{doubly commuting} if
\begin{equation}\label{eqn: dc def}
V_i^* V_j = V_j V_i^* \qquad (i \neq j).
\end{equation}
In particular, if $(V_1, V_2)$ is a pair of doubly commuting isometries, then there exist four closed joint $(V_1, V_2)$-reducing subspaces $\clh_A$, $A\subseteq \{1,2\}$, such that
\begin{equation}\label{eqn: Sloc decom}
\clh = \bigoplus_{A \subseteq \{1,2\}}\clh_{A},
\end{equation}
and $V_i|_{\clh_A}$ is a shift or unitary according to $i \in A$ or $i \notin A$, respectively. S{\l}oci\'{n}ski's observation (also see Suciu \cite{Suciu}) provided much of the inspiration for the subsequent development of orthogonal decompositions of tuples of isometries. For example, see \cite{Gaspar, Kosiek, S} for doubly commuting analogue in higher dimensions and other general settings.

Now, passing to general tuples of isometries, there are two instances that are relevant to us: Popovici \cite{Pop} proved that a pair of commuting isometries admits an orthogonal decomposition, just as in \eqref{eqn: Sloc decom} above, but with the additional property that the fourth part $(V_1, V_2)|_{\clh_{\{1,2\}}}$ is a weak bi-shift. Here (and similarly for general tuples of operators) $(V_1, V_2)|_{\clh_{\{1,2\}}}$ refers to the pair $(V_1|_{\clh_{\{1,2\}}}, V_2|_{\clh_{\{1,2\}}})$ on $\clh_{\{1,2\}}$. Weak bi-shifts are rather complicated and encode the information of the complexity of pairs of commuting isometries. Secondly, motivated by Heisenberg group $C^*$-algebras (and also \cite{JP}), in \cite{RSS} we introduced the notion of doubly twisted isometries (see Definition \ref{def: doub twist} below) and prove that a tuple of doubly twisted isometries always admits a von Neumann–Wold decomposition. Of course, doubly twisted isometries are fairly noncommutative objects, and therefore, it is curious to observe that the existence of the von Neumann–Wold decomposition is not completely a commutative feature. Note that doubly twisted isometries were referred to as tuples of $\clu_n$-twisted isometries in \cite{RSS}.

In this paper, we study orthogonal decompositions at a higher level. First, we propose a more general framework for noncommuting tuples of isometries and classify tuples admitting von Neumann–Wold decomposition. Our classification unifies all the existing results on von Neumann–Wold decompositions of tuple of isometries. It is worthwhile to mention that our results restricted to the commuting tuples of operators recover the classification of Gaspar and Suciu \cite{Gaspar}. In other words, we point out that the idea of Gaspar and Suciu works for tuples of noncommuting isometries, which is fairly relevant as the existence and subsequent (direct or indirect) applications of orthogonal decomposition to $C^*$-algebras appear to be fruitful (although we do not pursue this direction here). See \cite{Pelle, Palle 1, Weber 3, W}, and also see the central paper \cite{Marek}.

Secondly, we introduce the notion of twisted isometries and prove that an $n$-tuple of twisted isometries $V = (V_1, \ldots, V_n)$ admits an orthogonal decomposition with the additional property that $V|_{\clh_{\{1,\ldots, n\}}}$ is an $n$-tuple of twisted weak shift. This decomposition is more general than that of \cite{Pop} as well as \cite{RSS}. On one hand, our orthogonal decompositions for commuting tuples work for $n$-tuples, $n \geq 2$, and on the other hand, the generalizations of wandering subspaces and weak bi-shifts for $n$-tuples of isometries seem to be of interest as these ideas are not straight extensions of existing two variable theory.

Let us now turn to the technical part: Throughout the paper, when we refer to a \textit{twist} on $\clh$, if not otherwise mentioned, we mean $\binom{n}{2}$ commuting unitaries $\{U_{ij}\}_{1 \leq i < j \leq n}$ on $\clh$ such that
\[
U_{ji} := U_{ij}^* \qquad (1 \leq i < j \leq n).
\]
We interpret twists as a way of measurement of tuples of isometries from being commuting tuples. We now recall the notion of doubly twisted isometries which was introduced in \cite{RSS} and referred to as $\clu_n$-twisted isometry.

\begin{definition}[Doubly twisted isometries]\label{def: doub twist}
An $n$-tuple of isometries $(V_1, \ldots ,V_n)$ on $\clh$ is said to be doubly twisted with respect to a twist $\{U_{st}\}_{s<t} \subseteq \clb(\clh)$ if $V_k \in \{U_{st}\}_{s<t}'$ and
\[
V_i^*V_j=U_{ij}^*V_jV_i^*,
\]
for all $i,j,k=1, \ldots, n$, and $i \neq j$.
\end{definition}

We often suppress the twist $\{U_{ij}\}_{i<j}$ and simply say that $V$ is a doubly twisted isometry. In \cite{RSS}, we proved the existence of von Neumann–Wold decompositions for doubly twisted isometries. We further note that for a doubly twisted isometry $(V_1, \ldots, V_n)$, we necessarily have that $V_iV_j = U_{ij} V_j V_i$ for all $i \neq j$ (cf. \cite[Lemma 3.1]{RSS}). This motivates:

\begin{definition}[Twisted isometries]
An $n$-tuple of isometries $V = (V_1, \ldots ,V_n)$ on $\clh$ is said to be a twisted isometry with respect to a twist $\{U_{st}\}_{s<t} \subseteq \clb(\clh)$ if $V_k \in \{U_{st}\}_{s<t}'$ and
\[
V_iV_j=U_{ij} V_jV_i,
\]
for all $i,j,k=1, \ldots, n$, and $i \neq j$. The tuple $V$ is said to be a twisted isometry if it is a twisted isometry corresponding to some twist.
\end{definition}

The particular case $U_{ij} = I_{\clh}$ for all $i<j$ yields tuples of commuting isometries. As pointed out previously, in the case of pairs of commuting isometries, Popovici \cite{Pop} introduced orthogonal decompositions replacing S{\l}oci\'{n}ski's doubly commuting shift part with weak bi-shift. A pair of commuting isometries is called \textit{weak bi-shift} if
\[
V_1|_{\bigcap_{i \geq 0} \ker V_2^*V_1^i}, \; V_2|_{{\bigcap}_{j \geq 0} \ker V_1^*V_2^j}, \text{ and } V_1V_2,
\]
are shifts. In this context, we remark that a pair of shifts $(V_1, V_2)$ is doubly commuting if and only if $V_1|_{\ker V_2^*}$, $V_2|_{\ker V_1^*}$ and $V_1V_2$ are shifts \cite{Pop}. However, as we will see, for a general $n$-tuple, one needs a little more care in extending the ideas of Popovici. After some preparation on the geometric structure of twisted shifts, in Definition \ref{def: weak shift}, we introduce the notion of twisted weak shifts, a twisted counterpart of Popovici's weak bi-shifts. Finally, in Theorem \ref{thm: twisted dec final}, we prove that a twisted isometry admits an orthogonal decomposition with $V|_{\clh_{I_n}}$ as a twisted weak shift. In summary, an $n$-tuple of twisted isometry $V=(V_1, \ldots, V_n)$ gives rise to a direct sum decomposition just as in the case of doubly commuting isometry but the restriction of $V$ on $\clh_{I_n}$ is (the twisted weak shift) is rather complicated and encode the information of the complexity of tuples of twisted (or even commuting) isometries. In other words, the first $2^n-1$ subspaces of the direct summand are simple and enjoy similar properties to that of S{\l}oci\'{n}ski decomposition, and the Popovici analogue of the $2^n$-th summand is rather the challenging part.

The rest of the paper is organized as follows. In Section \ref{sec: charact wold decom}, we classify tuples of (not necessarily commuting) isometries that admit von Neumann–Wold decomposition. The existence of von Neumann–Wold decompositions of doubly twisted isometries then follows as a simple corollary.

Section \ref{sec: ortho dec} revisits von Neumann-Wold decompositions of doubly twisted isometries, which was first observed in  \cite{RSS}. The present proof follows as an application of the classification of von Neumann-Wold decompositions for general tuples of isometries. On one hand, this makes the paper an independent read and easy reference for the remaining results. On the other hand, this connects the recent developments with the techniques which were observed more than three decades ago by Gaspar and Suciu \cite{Gaspar}.

Section \ref{sec: wandering sub} sets the stage for wandering subspaces for twisted isometries. This section is central to the theory of orthogonal decomposition of twisted isometries.

In Section \ref{sec: weak shift}, we introduce the notion of twisted weak shifts, the twisted counterpart of Popovici's weak bi-shifts. The idea of twisted weak shifts comes from a classification of twisted shifts (see Proposition \ref{prop: n-shift}). Section \ref{sec: Popovici-Wold type decomposition} deals with orthogonal decompositions of twisted isometries.

\newsection{Characterizations of von Neumann-Wold Decompositions}\label{sec: charact wold decom}

In this section, we work in the category of noncommuting tuples of isometries. More specifically, we classify tuples of isometries that admit von Neumann–Wold decomposition.

Let $V \in \clb(\clh)$ be an isometry. To simplify the notation, we set
\begin{equation}\label{eqn: rep of Wold dec}
\clh_{V, u} = \bigcap_{m \in \Z_+} V^m \clh, \text{ and } \clh_{V, s} = \bigoplus_{m \in \Z_+} V^m (\ker V^*),
\end{equation}
the unitary part and the shift part, respectively, of $V$ (see Theorem \ref{thm: Wold}). Therefore, $\clh = \clh_{V, s} \oplus \clh_{V, u}$, where $V|_{\clh_{V, s}}$ is a shift and $V|_{\clh_{V, u}}$ is a unitary.

The following notation will be convenient throughout this paper:
\[
I_n = \{1, \ldots, n\}.
\]
Let $V = (V_1, \ldots, V_n)$ be an $n$-tuple of isometries acting on $\clh$. Recall from Definition \ref{def: orth decomp} that $V$ admits a von Neumann–Wold decomposition if there exist $2^n$ closed subspaces $\{\clh_A\}_{A \subseteq I_n}$ of $\clh$ such that $\clh = \bigoplus_{A \subseteq I_n} \clh_A$, and
\begin{enumerate}
\item $\clh_A$ reduces $V$ for all $A \subseteq I_n$, and

\item ${V_i}|_{\clh_A}$, $i \in A$, is a shift, and ${V_j}|_{\clh_A}$, $j \in A^c$, is a unitary for all $A \subseteq I_n$.
\end{enumerate}

We are now ready for the characterization of tuples of isometries admitting von Neumann-Wold decomposition.

\begin{thm}\label{thm: class of decom}
Let $V = (V_1, \ldots, V_n)$ be an $n$-tuple of isometries on $\clh$. The following are equivalent:

\begin{enumerate}
\item $V$ admits a von Neumann-Wold decomposition.
\item $\clh_{V_i,u}$ reduces $V_j$ for all $i, j \in I_n$.
\item $\clh_{V_i,s}$ reduces $V_j$ for all $i, j \in I_n$.
\end{enumerate}
\end{thm}
\begin{proof}
Evidently, it is enough to prove that (1) and (2) are equivalent. Let $V$ admits a von Neumann-Wold decomposition $\clh = \bigoplus_{A \subseteq I_n} \clh_A$. Fix $i \in I_n$. By the definition of von Neumann-Wold decomposition (or property (2) above) and Theorem \ref{thm: Wold}, we have
\[
\clh_{V_i, s} = \bigoplus_{\substack{B \subseteq I_n \\ i\in B}} \clh_B, \text{ and } \clh_{V_i, u} = \bigoplus_{\substack{B \subseteq I_n \\ i\notin B}} \clh_{B}.
\]
Since $\clh_A$ reduces $V_j$ for all $j \in I_n$ and $A \subseteq I_n$, it follows that $\clh_{V_i, u}$ reduces $V_j$ for all $i, j \in I_n$.

\noindent For the converse, suppose $\clh_{V_i,u}$ (and hence $\clh_{V_i,s}$ too) reduces $V_j$ for all $i, j \in I_n$. For each $A \subseteq I_n$, define
\begin{equation}\label{eqn: rep of Wold dec H A}
\clh_A = \Big[\bigcap_{i \in A} \clh_{V_i,s}\Big] \bigcap \Big[\bigcap_{j \in A^c} \clh_{V_j,u}\Big].
\end{equation}
Clearly, $\clh_A$ {reduces} $V_i$ for all $i \in I_n$ and $A \subseteq I_n$. It then remains to show that $\clh = \bigoplus_{A \subseteq I_n} \clh_A$. Since $\bigoplus_{A \subseteq I_n} \clh_A \subseteq \clh $, it is enough to prove that
\[
\clh = \clh_{V_1,s} \oplus \clh_{V_1, u} \subseteq \bigoplus_{A \subseteq I_n} \clh_A.
\]
For this, we need a general observation: Given $A\subseteq I_m \subsetneqq I_n$ and $j\notin I_m$, we denote by $\tilde{A}$ the set $A$ itself but as a subset of $I_m \cup \{j\}$. We claim that
\[
\clh_A \subseteq \clh_{\tilde{A} \cup \{j\}} \cup \clh_{\tilde{A}}.
\]
Note that $\clh_A$ reduces $V_j$, and also $\ker V_j^*|_{\clh_A} = \ker V_j^* \cap \clh_A$. By applying the von Neumann-Wold Decomposition to the isometry $V_j|_{\clh_A}$, we have
\[
\clh_A = \Big[\underset{k_j \in \Z_+}{\bigoplus} V_j^{k_j} (\ker V^*_j \cap \clh_A)\Big]
\bigoplus
\Big[\underset{k_j \in \Z_+}{\bigcap V_j}^{k_j} \clh_A \Big].
\]
For each $k_j \in \mathbb{Z}_{+}$, it is obvious that
\[
V_j^{k_j} (\ker V^*_j \cap \clh_A)
\subseteq V_j^{k_j} \clh_A \subseteq \clh_A,
\]
and
\[
V_j^{k_j} (\ker V^*_j \cap \clh_A) \subseteq  V_j^{k_j} (\ker V^*_j).
\]
The latter inclusion yields
\[
\bigoplus_{k_j \in \Z_+} V^{k_j}_j (\ker V^*_j \cap \clh_A) \subseteq \bigoplus_{k_j \in \Z_+} V_j^{k_j} (\ker V_j^*) = \clh_{V_{j,s}},
\]
which along with the former inclusion gives
\[
\bigoplus_{k_j \in \Z_+} V_j^{k_j} (\ker V^*_j \cap \clh_A) \subseteq
\clh_A \cap \clh_{V_j, s} = \clh_{\tilde{A} \cup \{j\}}.
\]
We also have
\[
\underset{k_j \in \Z_+}{\bigcap} V_{j}^{k_j} \clh_A \subseteq \clh_A  \bigcap
\Big(\underset{k_j \in \Z_+}{\bigcap} V_j^{k_j} \clh \Big) = \clh_{\tilde{A}}.
\]
This implies $\clh_A \subseteq \clh_{\tilde{A} \cup \{j\}} \oplus \clh_{\tilde{A}}$ and proves the claim. Applying this to $A\subseteq I_m \subsetneqq I_n$ and $j, k \notin I_m$, we obtain
\[
\clh_A \subseteq \clh_{\tilde{A} \cup \{j\}\cup \{k\}}\oplus \clh_{\tilde{A} \cup \{j\}}\oplus \clh_{\tilde{A} \cup \{k\}} \oplus\clh_{\tilde{A}},
\]
where $\tilde{A} = A$ but a subset of $I_m \cup \{j, k\}$. Consider the von-Neumann Wold decomposition for $V_1$ on $\clh$:
\[
\clh = \clh_{V_1, s} {\oplus} \clh_{V_1, u} = \clh_{\{1\}} \oplus \clh_{\emptyset},
\]
where index sets $\{1\}$ and $\emptyset$ on the right side are subsets of $I = \{1\}$. Applying the above recipe repeatedly to $\clh_{\{1\}}$ one sees that
\[
\clh_{\{1\}} \subseteq \underset{A \subseteq J}{\oplus} \clh_{\tilde{\{1\}} \cup {A}},
\]
where $J = \{2,\ldots,n\}$, and $\tilde{\{1\}} = \{1\}$ but a subset of $I_n$. Similarly, $\clh_{\emptyset} \subseteq \underset{A \subseteq J}{\bigoplus} \clh_{A}$. Then
\[
\begin{split}
\clh & = \clh_{\{1\}} \oplus \clh_{\emptyset}
\\
& \subseteq \big(\underset{A \subseteq J}{\oplus} \clh_{\tilde{\{1\}} \cup A}\big) \oplus \big(\underset{A \subseteq J}{\oplus} \clh_A \big)
\\
& = \underset{A \subseteq I_{n}}{\bigoplus} \clh_{A},
\end{split}
\]
completes the proof of the theorem.
\end{proof}

In the case of commuting tuples of isometries, the above result was stated by Gaspar and Suciu (see \cite[Theorem 2]{Gaspar}). The proof in \cite{Gaspar} was mentioned only in the case of $n=3$. It is very curious to observe that the commutative theme of Gaspar and Suciu also works for noncommuting tuples of isometries.

Two particular cases of the above classification are worthy of special attention: doubly commuting tuples of isometries (also see \cite{Gaspar}) and doubly twisted isometries. Of course, the latter notion is more general than the former. We conclude this section with the case of doubly commuting isometries.

Before proceeding further, it is useful to make some standard observations about isometries. The proof follows from representations of unitary and shift parts of isometries as in \eqref{eqn: rep of Wold dec}. Given a closed subspace $\cls \subseteq \clh$, denote by $P_{\cls}$ the orthogonal projection of $\clh$ onto $\cls$.

\begin{lem}\label{lem: iso and proj}
Let $V \in \clb(\clh)$ be an isometry. Then
\begin{enumerate}
\item $P_{\clh_{V, s}} = \text{SOT}-\sum_{m \in \Z_+} V^{m} P_{\ker V^*} V^{*m}$.
\item $P_{\clh_{V, u}} = \text{SOT}-\lim_{m \raro \infty} V^m V^{*m}$.
\end{enumerate}
\end{lem}

Let $V = (V_1, \ldots, V_n)$ be a tuple of doubly commuting isometries on $\clh$ (see \eqref{eqn: dc def}). By Lemma \ref{lem: iso and proj}, it follows that
\[
P_{\clh_{V_i, u}} = \text{SOT}-\lim_{m \raro \infty}V_i^m V_i^{*m} \qquad (i\in I_n).
\]
In particular, if $h \in \clh$ and $i \in I_n$, then $h \in \clh_{V_i,u}$ if and only if $P_{\clh_{V_i, u}} h = h$, whereas $h \in \clh_{V_i, s}$ if and only if $P_{\clh_{V_i, u}} h = 0$. Clearly, $\clh_{V_i, u}$ reduces $V_i$ for all $i \in I_n$. Moreover, for all $i \neq j$ in $I_n$, we have
\[(V^m_i V^{*m}_i) V_j = V_j (V^m_i V^{*m}_i)
\]
that is
\[
P_{\clh_{V_i, u}} V_j = V_j P_{\clh_{V_i, u}}.
\]
This implies immediately that $\clh_{V_i,u}$ reduces $V_j$ for all $i, j \in I_n$. Theorem \ref{thm: class of decom} then implies that $V$ admits von Neumann-Wold decomposition. This assertion was observed earlier in \cite{S}. The assertion also follows from Corollary \ref{cor: dti admits od}.

It is now an interesting problem to represent the direct summands $\clh_A$ of the von Neumann-Wold decomposition as described in \eqref{eqn: rep of Wold dec H A}. The answer is not clear in this generality and even it is unclear what conditions we would need to impose to get concrete representations of $\clh_A$. In the following section, we will discuss this in the setting of doubly twisted isometries.

\newsection{Doubly twisted isometries}\label{sec: ortho dec}

In this section, we deal with doubly twisted isometries (see Definition \ref{def: doub twist}). The goal here is to recover, as an application of Theorem \ref{thm: class of decom}, the existence of the von Neumann-Wold decomposition for doubly twisted isometries along with representations of $\clh_A$'s as described in \eqref{eqn: rep of Wold dec H A}. Needless to say, the main results of this section are not new (cf. \cite{RSS}), but the techniques involved are more algebraic. We believe that the present approach has more potential in other general frameworks.

Let $(V_1, \ldots, V_n)$ be a doubly twisted isometry on $\clh$. We know that
\begin{equation}\label{eqn: doubly twisted cond}
V_i^* V_j = U_{ij}^* V_j V_i^*, \text{ and } V_i V_j = U_{ij} V_j V_i,
\end{equation}
for all $i \neq j$ and $i, j \in I_n$. For each $i \neq j$, we then have
\begin{equation}\label{eqn: doubly twisted Uij}
U_{ij} = V_i^* V_j^* V_i V_j,
\end{equation}
and
\[
\begin{split}
V_i(V_jV_j^*)&  = U_{ij} V_j V_i V_j^* = U_{ij}V_j (U_{ij}^* V_j^* V_i) = (V_j V_j^*) V_i.
\end{split}
\]
In the above, we have used the fact that $V_p \in \{U_{st}\}_{s<t}'$ for all $p \in I_n$, and $U_{st}^* = U_{ts}$, $t > s$. Therefore
\begin{equation}\label{eqn: Vi VjVj*}
V_i (V_j V_j^*) = (V_j V_j^*) V_i \qquad (i \neq j).
\end{equation}
Let $\cln_{\emptyset} = \clh$, and let
\[
\cln_A = \bigcap_{i \in A} \ker V_i^* \qquad (A \neq \emptyset).
\]
Also set $\cln_i = \cln_{\{i\}} = \ker V_i^*$ for all $i \in I_n$. In view of the above observation, we have
\begin{equation}\label{eqn: ViVi* comm VjVj*}
(V_iV_i^*) (V_j V_j^*) = (V_j V_j^*) (V_i V_i^*) \qquad (i \neq j).
\end{equation}
Since $I - V_iV_i^* = P_{\cln_i}$ for all $i \in I_n$, it follows that $\{P_{\cln_i}\}_{i \in I_n}$ is a family of commuting orthogonal projections. This implies
\[
P_{\cln_A} = \prod_{i \in A} P_{\cln_i} \qquad (A \subseteq I_n, A \neq \emptyset).
\]
Moreover, \eqref{eqn: Vi VjVj*} implies that $\cln_A$ reduces $V_j$ for all $j \notin A$. Since $V_i \in \{U_{st}\}_{s < t}'$, it follows that $(I - V_i V_i^*) U_{st} = U_{st} (I - V_i V_i^*)$ for all $i \in I_n$ and hence, by the factorization of $P_{\cln_A}$ above, we have that $\cln_A$ reduces $U_{st}$ and $U_{st} \cln_A = \cln_A$ for all $s \neq t$ and $A \subseteq I_n$. We summarize all these observations in the following lemma (also see Lemma 3.3 and Lemma 3.5 in \cite{RSS}).

\begin{lem}\label{lem: d twist basic lemma}
Let $(V_1, \ldots, V_n)$ be a doubly twisted isometry, and let $A \subseteq I_n$. Then the following statements hold:
\begin{enumerate}
\item $\{P_{\cln_i}\}_{i \in I_n}$ is a family of commuting orthogonal projections.
\item $P_{\cln_A} = \prod_{i \in A} P_{\cln_i}$, $A \neq \emptyset$.
\item $\cln_A$ reduces $V_j$ for all $j \notin A$, and $A \neq I_n$.
\item $\cln_A$ reduces $U_{st}$ and $U_{st} \cln_A = \cln_A$ for all $s \neq t$.
\end{enumerate}
\end{lem}

We are now ready for the second application of Theorem \ref{thm: class of decom}:

\begin{cor}\label{cor: dti admits od}
Doubly twisted isometries admit von Neumann-Wold decomposition.
\end{cor}
\begin{proof}
Let $(V_1, \ldots, V_n)$ be a doubly twisted isometry on $\clh$. Fix $i, j \in I_n$, and suppose $i \neq j$. Since $V_i V_j^m = V_j^m V_i U_{ij}^m$ for all $m \in \Z_+$, and $U_{ij} \clh = \clh$ and $V_i \clh \subseteq \clh$, it follows that
\[
V_i \clh_{V_j, u} = V_i (\bigcap_{m \in \Z_+} V^m_j \clh )
=  \bigcap_{m \in \Z_+} V^m_j V_i U^m_{ij} \clh
\subseteq  \bigcap_{m \in \Z_+} V^m_j  \clh
\]
that is, $V_i \clh_{V_j, u} \subseteq \clh_{V_j, u}$. Similarly, $V_i^* \clh_{V_j, u} \subseteq \clh_{V_j, u}$. This proves that $\clh_{V_j, u}$ reduces $V_i$. Finally, the fact that $\clh_{V_i, u}$ reduces $V_i$ follows from the von Neumann-Wold decomposition of $V_i$. Thus the assertion follows from Theorem \ref{thm: class of decom}.
\end{proof}

Now we turn to the problem of representations of the direct summands of the von Neumann-Wold decomposition. First, observe that Lemma \ref{lem: d twist basic lemma} and part (2) of Lemma \ref{lem: iso and proj} implies:

\begin{lem} \label{commutativity 1}
Let $(V_1, \ldots, V_n)$ be a doubly twisted isometry, and let $A \subsetneqq I_n$ be a nonempty subset. Then $P_{\clh_{V_j, s}}, P_{\clh_{V_j, u}} \in \{P_{\cln_A}\}'$ for all $j \in A^c$.
\end{lem}

We then have the commutativity of orthogonal projections:

\begin{lem} \label{Commutativity}
Let $(V_1, \ldots, V_n)$ be a doubly twisted isometry. Then $\{P_{\clh_{V_i},s}, P_{\clh_{V_j},u}\}_{i, j \in I_n}$ is a family of commuting orthogonal projections.
\end{lem}
\begin{proof}
Clearly, $P_{\clh_{V_i,s}} P_{\clh_{V_i,u}} = 0 = P_{\clh_{V_i,u}} P_{\clh_{V_i,s}}$ for all $i \in I_n$. Therefore, assume that $i \neq j$. Lemma \ref{lem: iso and proj} and Lemma \ref{commutativity 1} then imply
\[
\begin{split}
P_{\clh_{V_i,s}} P_{\clh_{V_j,s}}
& = (\sum_{m \in \Z_+} V_i^m P_{\cln_i} V^{*m}_i) P_{\clh_{V_j,s}}
\\
& = P_{\clh_{V_j,s}} (\sum_{m \in \Z_+} V_i^m P_{\cln_i} V^{*m}_i)
\\
& = P_{\clh_{V_j,s}} P_{\clh_{V_i,s}}.
\end{split}
\]
In the above, we have also used the fact that $\clh_{V_j, s}$ reduces $V_i$. A similar computation then yields $P_{\clh_{V_i,s}} P_{\clh_{V_j,u}} = P_{\clh_{V_j,u}} P_{\clh_{V_i,s}}$ and $P_{\clh_{V_i,u}} P_{\clh_{V_j,u}} = P_{\clh_{V_j,u}} P_{\clh_{V_i,u}}$, and completes the proof of the lemma.
\end{proof}

Similar technique applies to the splitting of product of isometries, co-isometries, and orthogonal projections. For each $m \geq 1$, we denote by $k = (k_1, \ldots, k_m)$ the multi-index in $\Z_+^m$.

\begin{lem}\label{Simplify}
Let $(V_1, \ldots, V_n)$ be a doubly twisted isometry, $A \subseteq I_n$, $A \neq \emptyset$, and let $k \in \Z_+^{|A|}$. Then
\[
\prod_{i \in A} \Big(V_i^{k_i} P_{\cln_i} V_i^{*k_i}
\Big) = \Big(\prod_{i \in A} V_i^{k_i} \Big) P_{\cln_A} \Big(\prod_{i \in A} V_i^{ k_i} \Big)^*.
\]
\end{lem}
\begin{proof}
Let $i, j \in I_n$, and suppose $i \neq j$. Note that (see \eqref{eqn: doubly twisted cond})
\[
V_i^{*p} V_j^q = U_{ij}^{*pq} V_j^q V_i^{*p} \qquad (p,q \in \Z_+).
\]
Since $U_{ij} \in \{V_i, V_j\}'$ and $\cln_s = \ker V_s^*$ reduces $V_t$ for all $t \neq s$ (see part (3) of Lemma \ref{lem: d twist basic lemma}), for each $k_i, k_j \in \Z_+$, it follows that
\[
\begin{split}
(V_i^{k_i} P_{\cln_i} V_i^{*k_i}) (V_j^{k_j} P_{\cln_j} V_j^{*k_j}) & = (U_{ij}^{*k_ik_j}) V_i^{k_i} P_{\cln_i}  V_j^{k_j} V_i^{*k_i} P_{\cln_j} V_j^{*k_j}
\\
& = (U_{ij}^{*k_ik_j}) V_i^{k_i} V_j^{k_j} P_{\cln_i} P_{\cln_j} V_i^{*k_i} V_j^{*k_j}
\\
& = (U_{ij}^{*k_ik_j}) (V_i^{k_i} V_j^{k_j}) P_{\cln_{\{i,j\}}} (V_i^{k_i} V_j^{k_j})^* (U_{ij}^{k_ik_j})
\\
& = (V_i^{k_i} V_j^{k_j}) P_{\cln_{\{i,j\}}} (V_i^{k_i} V_j^{k_j})^*.
\end{split}
\]
This verifies the conclusion for $A = \{i,j\}$. The lemma now follows by induction.
\end{proof}

We set the following convention. Given an $n$-tuple of bounded linear operators $T = (T_1, \ldots, T_n)$ on $\clh$ and $k = (k_1, \ldots, k_n) \in \Z_+^n$, we define $T^k$ by
\[
T^k:= T_{1}^{k_1}\cdots T_{n}^{k_n}.
\]
Now we discuss the unitary part of doubly twisted isometries:

\begin{lem}\label{twisted intersection}
Let $V = (V_1, \ldots, V_n)$ be a doubly twisted isometry and let $\cls$ reduces $V$. Then
\[
\underset{k \in \mathbb{Z}_+^n}{\bigcap} V^k \cls = \underset{i \in I_n}{\bigcap} \underset{k_i \in \mathbb{Z}_+}{\bigcap}
V_i^{k_i} \cls.
\]
\end{lem}
\begin{proof}
It suffices to prove the assertion for $n=2$; the general case then easily follows by induction. Let $V=(V_1, V_2)$ be a doubly twisted isometry and let $U$ be the corresponding twist. By \eqref{eqn: doubly twisted cond} and \eqref{eqn: doubly twisted Uij}, we know that $V_1 V_2 = U V_2 V_1$ and $U = V_1^*V_2^* V_1 V_2$. In particular, $\cls$ reduces $U$. It is now easy to see that
\[
\underset{k \in \mathbb{Z}_+^2}{\bigcap} V^k S \subseteq \underset{i \in I_2}{\bigcap} \underset{k_i \in \mathbb{Z}_+}{\bigcap} V_i^{k_i} \cls.
\]
For the reverse inclusion, let $y \in \cap_{k_1 \in \mathbb{Z}_+} V_1^{k_1} \cls \cap_{k_2 \in \mathbb{Z}_+} V_2^{k_2} \cls$. For each $k_1, k_2 \in \mathbb{Z}_+$, there exist $x_{k_1}$ and $x_{k_2}$ in $\cls$ such that
\[
y = V_1^{k_1} x_{k_1} = V_2^{k_2} x_{k_2}.
\]
Then
\[
x_{k_1} = V^{* k_1}_1 V^{k_2}_2 x_{k_2} =
V_2^{k_2} (V_1^{*k_1} U^{*k_1 k_2} x_{k_2}) = V_2^{k_2} x,
\]
where $x=V_1^{*k_1} U^{*k_1 k_2} x_{k_2} \in \cls$. Therefore, $y = V_1^{k_1} V_2^{k_2} x \in \cap_{k \in \mathbb{Z}_+^2} V^k \cls$, which completes the proof of the lemma.
\end{proof}

The following conventions will be in effect throughout: For an $n$-tuple of bounded linear operators $T = (T_1, \ldots, T_n)$ on $\clh$ and $A = \{m_1<\cdots <m_p\}\subseteq I_n$, we define
\[
T_A = (T_{m_1}, \ldots, T_{m_p}).
\]
For each $k = (k_1, \ldots, k_p) \in \Z_+^p$, we define $T_A^k$ by
\begin{equation}\label{eqn: def of TA}
T_A^k:= T_{m_1}^{k_1}\cdots T_{m_p}^{k_p}.
\end{equation}
Now we are ready to compute the closed subspaces $\clh_A$ in the direct summands \eqref{eqn: rep of Wold dec H A} of von Neumann-Wold decomposition of doubly twisted isometries.

\begin{thm}\label{thm: Wold for Un twisted}
Let $V = (V_1, \ldots, V_n)$ be a doubly twisted isometry. Then $V$ admits a von Neumann-Wold decomposition $\clh = \bigoplus_{A \subseteq I_n} \clh_A$, where
\[
\clh_A = \bigoplus_{k \in \Z_+^{|A|}} V_A^k \Big(\bigcap_{l \in \Z_+^{n-|A|}}V^{l}_{I_n \setminus A} \cln_A\Big) \qquad (A \subseteq I_n).
\]
\end{thm}
\begin{proof}
Note, by Corollary \ref{cor: dti admits od}, we already know that doubly twisted isometries admit von Neumann-Wold decomposition. Therefore, we only have to prove the representation of $\clh_A$, where we know that (see \eqref{eqn: rep of Wold dec H A} in the proof of Theorem \ref{thm: class of decom})
\[
\clh_A = \Big(\bigcap_{i \in A} \clh_{V_i,s}\Big) \bigcap \Big(\bigcap_{j \in A^c} \clh_{V_j,u}\Big).
\]
Fix $A \subseteq I_n$, and suppose $A \neq \emptyset$. Lemma \ref{Commutativity} implies $\clh_A = \Big(\underset{i \in A}{\prod} P_{\clh_{V_i, s}}\Big) \Big(\underset{j \in A^c}{\prod} P_{\clh_{V_j, u}} \Big)\clh$. We compute
\[
\begin{split}
\prod_{i \in A} P_{\clh_{V_i,s}} & = \prod_{i \in A} \Big(\sum_{k_i \in \mathbb{Z}_+} V_i^{k_i} P_{\cln_i} V_i^{*k_i}\Big)
\\
& = \sum_{k \in \Z_+^{|A|}}\Big(\prod_{i \in A} V_i^{k_i} P_{\cln_i} V_i^{* k_i}\Big)
\\
& = \sum_{k \in \Z_+^{|A|}}\Big(\prod_{i \in A} V_i^{k_i}\Big) P_{\cln_A} \Big(\prod_{i \in A} V_i^{k_i}\Big)^*,
\end{split}
\]
where the last equality follows from Lemma \ref{Simplify}. If $A=I_n$, then the above equality gives desired representation of $\clh_{I_n}$. Suppose $A \subsetneq I_n$. By Lemma \ref{commutativity 1}, we know that $P_{\clh_{V_j, u}} P_{\cln_A} = P_{\cln_A} P_{\clh_{V_j,u}}$ for all $j \in A^c$. Also, since $P_{\clh_{V_j, u}} V_i = V_i P_{\clh_{V_j, u}}$ for all $i \in A$ and $j\in A^c$, we have
\[
\begin{split}
\Big(\prod_{i \in A} P_{\clh_{V_i,s}}\Big) \Big(\prod_{j \in A^c}P_{\clh_{V_j,u}} \Big) & = \Big(\sum_{k \in \Z_+^{|A|}} \Big(\prod_{i \in A} {V_{i}^{k_i}} \Big)
P_{\cln_A} \Big(\prod_{i\in A} V_{i}^{k_i} \Big)^{*} \Big)\Big(\prod_{j \in A^c} P_{\clh_{V_j,u}} \Big)
\\
& = \sum_{k \in \Z_+^{|A|}} \Big(\prod_{i \in A} {V_{i}^{k_i}}
\Big) \Big(\prod_{j \in A^c} P_{\clh_{V_j,u}} \Big) P_{\cln_A} \Big(\prod_{i \in A} V_{i}^{k_i}
\Big)^*.
\end{split}
\]
Consequently
\[
\clh_A = \text{ran } \Big(\sum_{k \in \Z_+^{|A|}} \Big(\prod_{i \in A} {V_{i}^{k_i}} \Big) \Big(\prod_{j \in A^c} P_{\clh_{V_j,u}} \Big) P_{\cln_A}\Big).
\]
Finally, note that
\[
\begin{split}
\text{ran } \Big(\Big(\prod_{j \in A^c} P_{\clh_{V_j,u}} \Big) P_{\cln_A}\Big) & = \bigcap_{j \in A^c} P_{\clh_{V_j,u}} \cln_A
\\
& = \bigcap_{j \in A^c}
\Big(\bigcap_{k_j \in \Z_+}V_j^{k_j} (\cln_A) \Big)
\\
& = \bigcap_{l \in \Z_+^{n-|A|}} {V_{I_n\setminus A}^l}(\cln_A),
\end{split}
\]
where the last equality follows from Lemma \ref{twisted intersection}. Therefore
\[
\clh_A  = \bigoplus_{k \in \Z_+^{|A|}} V_A^k\Big( \underset{l \in \Z_+^{n-|A|}}{\bigcap} {V_{I_n \setminus A}^l}(\cln_A)\Big).
\]
If $A = \emptyset$, then $\cln_A = \clh$, and hence
\[
\bigcap_{l \in \Z_+^{n}} {V_{I_n}^l} \cln_{\emptyset} = \bigcap_{l \in \Z_+^{n}} {V_{I_n}^l} \clh = \text{ran} \Big(\prod_{i \in I_n} P_{\clh_{V_i,u}}\Big) = \bigcap_{i \in I_n} \clh_{V_i,u},
\]
which completes the proof of the theorem.
\end{proof}

As already pointed out, the above decomposition was first observed in \cite[Theorem 3.6]{RSS}. The present proof is conceptually different and presented as an application of Theorem \ref{thm: class of decom}. Moreover, the techniques and ideas used in this section will be useful in describing the structure of twisted isometries.

\newsection{Wandering subspaces}\label{sec: wandering sub}

This section is devoted to the study of wandering subspaces. Let $V = (V_1, \ldots, V_n)$ be a tuple of (not necessarily commuting) isometries on $\clh$, and let $\cls$ be a closed subspace of $\clh$. We say that $\cls$ satisfies \textit{wandering property} (or $\cls$ is a wandering subspace) for $V$ if
\[
V^k \cls \perp V^l \cls,
\]
for all $k\neq l$ in $\Z_+^n$. If, in addition
\[
\clh = \bigoplus_{k \in \Z^n_+} V^k \cls,
\]
then we say that $\cls$ is a \textit{generating wandering subspace} for $V$. Generating wandering subspaces play a vital role in representing single isometries. Indeed, if $n=1$, then $\ker V^*$ is the generating wandering subspace which represents the pure part of an isometry $V$ (see Theorem \ref{thm: Wold}). Wandering subspaces for doubly twisted isometries are also explicit and plays important role in the structure of such tuples (compare Theorem \ref{thm: Wold for Un twisted} with the identity \eqref{eqn: W_A repr fod dc}). However, as we will see in this section, the analog of wandering subspaces for twisted isometries requires some more care. We begin with the following definition (recall the notaion \eqref{eqn: def of TA}).

\begin{definition}\label{def: weak wandering}
Let $V = (V_1, \ldots, V_n)$ be a tuple of isometries on $\clh$. The weak $A$-wandering subspace for $V$ is defined by
\[
\mathcal{E}_A:=\bigcap_{i\in A}	\bigcap_{B \subseteq \{i\}^c}
\bigcap_{k\in \Z_+^{|B|}} \Big(\ker V_i^*V_B^k\Big),
\]
whenever $A \subseteq I_n$ and $A \neq \emptyset$, whereas  $\cle_{\emptyset}:=\clh$.
\end{definition}

In the above, $\{i\}^c := I_n \setminus \{i\}$. First, we clarify that weak $A$-wandering subspaces of twisted isometries indeed satisfy wandering property.

\begin{lem}{\label{lemma 2}}
Let $V = (V_1, \ldots, V_n)$ be a twisted isometry, and let $A$ be a nonempty subset of $I_n$. Then $\cle_A$ satisfies the wandering property for $V$.
\end{lem}
\begin{proof}
Let $|A| = p$ and suppose $A:=\{m_1<\cdots <m_p\}$. Let $k = (k_1,\ldots, k_p)$ and $l=(l_1,\ldots, l_p)$ be in $\Z_+^p$. Suppose $k \neq l$, and let $j$ be the minimum of all $i \in \{1,\cdots, p\}$ such that $k_i \neq l_i$. Assume, without loss of generality, that $k_j < l_j$. For $x, y \in \cle_A$, we have
\[
\begin{split}
\langle V_A^k x, V_A^l y \rangle & = \langle V_{m_j}^{k_j} \cdots V_{m_p}^{k_p} x, V_{m_j}^{l_j} \cdots V^{l_p}_{m_p}y \rangle
\\
& = \langle V_{m_j}^* (V_{m_{j+1}}^{k_{j+1}} \cdots V_{m_p}^{k_p}) x, V_{m_j}^{l_j- k_j - 1} \cdots V^{l_p}_{m_p}y \rangle.
\end{split}
\]
Note that (see Definition \ref{def: weak wandering})
\[
x \in \cle_A = \bigcap_{i\in A}	\bigcap_{B \subseteq \{i\}^c}
\bigcap_{k\in \Z_+^{|B|}} \Big(\ker V_i^*V_B^k\Big).
\]
Therefore, if we set $B:= \{m_{j+1}, \ldots, m_p\}$, then $B \subseteq \{m_j\}^c$, and consequently
\[
x \in \ker \Big(V_{m_j}^* V_B^{\tilde{k}} \Big),
\]
where $\tilde{k} = (k_{j+1}, \ldots k_p) \in \Z_+^{|B|}$. Consequently, $V_{m_j}^* V_B^{\tilde{k}} x = 0$, which implies
\[
\begin{split}
\langle V_A^k x, V_A^l y \rangle & = \langle V_{m_j}^* V_B^{\tilde{k}} x, V_{m_j}^{l_j- k_j - 1} \cdots V^{l_p}_{m_p}y \rangle
\\
& = 0,
\end{split}
\]
and completes the proof of the lemma.
\end{proof}

Before we get into the working class subspaces satisfying wandering property, we prove some lemmas. In the remaining part of this section we assume that $V = (V_1, \ldots, V_n)$ is a twisted isometry corresponding to a twist $\{U_{ij}\}_{i < j}$. The following simple observation will turn out to be an indispensable tool in what follows.

\begin{lem}\label{lemma: twist commute}
Let $k\in \Z_+^n$ and $i \in I_n$. Then there exists a monomial $\eta_{i,k} \in \C[z_1, \ldots, z_n]$ such that
\[
V_i V^k = V^k V_i \eta_{i,k}(U).
\]
\end{lem}
\begin{proof}
By the definition of twisted isometries, we have $V_i V_j = U_{ij} V_j V_i$ for all $i \neq j$. Therefore
\[
\begin{split}
V_i V^k & = (U^{k_1}_{i1} \cdots U^{k_{i-1}}_{i(i-1)}) (U^{k_{i+1}}_{i (i+1)} \cdots U^{k_n}_{i n}) V^k V_i
\\
& = \eta_{i,k}(U) V^k V_i
\\
& = V^k V_i \eta_{i,k}(U),
\end{split}
\]
where
\[
\eta_{i,k} = z_1^{k_1} \cdots z_{i-1}^{k_{i-1}} z_{i+1}^{k_{i+1}}\cdots z_n^{k_n} \in \C[z_1, \ldots, z_n],
\]
and $\eta_{i,k}(U)$ refers to the polynomial functional calculus.
\end{proof}

In the above, the polynomial functional calculus $\eta_{i,k}(U)$ is given by
\[
\eta_{i,k}(U) = U^{k_1}_{i1} \cdots U^{k_{i-1}}_{i(i-1)} U^{k_{i+1}}_{i (i+1)} \cdots U^{k_n}_{i n}.
\]
This is a convention we will adopt throughout the remainder of this paper. Moreover, if $i$ and $k$ are clear from the context, then we simply denote a monomial in $U$ by $\eta(U)$ instead of $\eta_{i,k}(U)$.

Before proceeding, we observe the useful identity
\begin{equation}\label{eqn: U = VV}
U_{ij} = V_i^* V_j^* V_i V_j \qquad (i \neq j).
\end{equation}
This follows from the definition of twisted isometries that $V_iV_j = U_{ij} V_j V_i$ and the fact that $V_i \in \{U_{st}\}_{s<t}'$ for all $i \neq j$.

\begin{lem}\label{lemma 1}
$U_{ij} \cle_A= \cle_A$ and $U_{ij} \cln_A= \cln_A$ for all $i\neq j$ and $A \subseteq I_n$.
\end{lem}
\begin{proof}
If $A=\emptyset$, then $\cle_\emptyset=\clh=\cln_\emptyset$, and the desired equality is clear. Suppose $A \neq \emptyset$. By assumption, we have
\begin{equation}\label{eqn: V i Vj commute}
U_{ij} V_k = V_k U_{ij}, \text{ and } U_{ij} V_k^* = V_k^* U_{ij},
\end{equation}
for all $i, j, k \in I_n$ and $i \neq j$. Fix $A \subseteq I_n$, $i \neq j$, and suppose $x \in \cle_A$. We know that $x \in \ker V_t^* V_B^k$ for all $t \in A$, $B \subseteq \{t\}^c$ and $k \in \Z_+^{|B|}$.
Therefore,
\[
V_t^* V_B^k U_{ij} x = U_{ij} V_t^* V_B^k x =0,
\]
and similarly $V_t^* V_B^k U_{ij}^*x = 0$. This implies that $U_{ij}$ reduces $\cle_A$. Since $U_{ij}$ is unitary, it follows that $U_{ij} \cle_A= \cle_A$. The second equality $U_{ij} \cln_A= \cln_A$ follows from \eqref{eqn: V i Vj commute}.
\end{proof}

Since $V_i V_j = U_{ij}V_j V_i = V_j V_i U_{ij}$ for all $i \neq j$, the above lemma implies that:

\begin{lem}\label{lemma: ViVj WA}
$V_iV_j \cle_A = V_j V_i \cle_A$ for all $A\subseteq I_n$.
\end{lem}

We also have the following invariance property:

\begin{lem}\label{lemma:Ea invariant}
$V_j \mathcal{E}_A \subseteq \mathcal{E}_A$ for all $A\subsetneqq I_n$ and $j\in A^c$.
\end{lem}
\begin{proof}
We know that $\mathcal{E}_\emptyset = \clh$. Then obviously $V_j \mathcal{E}_\emptyset \subseteq \mathcal{E}_\emptyset$. Suppose $A \neq \emptyset$. Let $ x\in \mathcal{E}_A$. Since $j\in A^c$, for any $B \subseteq \{i\}^c, k\in \Z_+^{|B|}$ and $i\in A$, we have, $V_i^*V_B^k V_j(x)=0$. This implies $V_j \mathcal{E}_A \subseteq \mathcal{E}_A$ and completes the proof of the lemma.
\end{proof}

Although the weakly wandering subspace $\cle_A$ satisfies the wandering property, for the sake of appropriate orthogonal decompositions of twisted isometries, we require to identify a suitable subspace of $\cle_A$:

\begin{definition}\label{def: wand sub}
Let $V = (V_1, \ldots, V_n)$ be a twisted isometry. For each $A \subseteq I_n$, the $A$-wandering subspace for $V$ is defined by
\[
\clw_A = \underset{l\in \Z_+^{n-|A|}}{\bigcap}V^l_{I_n\setminus A}\mathcal{E}_A.
\]
\end{definition}

Clearly, $\clw_A \subseteq \cle_A$ for all $A \subseteq I_n$. In particular, we have $V_A^k \clw_A \perp V_A^l \clw_A$ for all $k \neq l$ in $\Z_+^{|A|}$, $A \neq \emptyset$, and $A \subseteq I_n$. Therefore, the orthogonal sum
\begin{equation}\label{eqn: H VA}
\clh_{V, A}:= \underset{k\in \Z_+^{|A|}}{\bigoplus}V_A^k \clw_A,
\end{equation}
is well-defined for all $A \neq \emptyset$. We also set $\clh_{V, \emptyset} := \clw_{\emptyset}$, that is
\[
\clh_{V, \emptyset} := \clw_{\emptyset}=\underset{k\in \Z_+^{n}}{\bigcap}V^k_{I_n}\clh.
\]
By the commutativity property of $V_k$'s and $U_{ij}$'s as in \eqref{eqn: V i Vj commute} and Lemma \ref{lemma 1}, it follows that
\begin{equation}\label{eqn: Uij H VA}
U_{ij} \clh_{V,A} = \clh_{V,A},
\end{equation}
for all $i\neq j$ and $A \subseteq I_n$. Moreover, we have:

\begin{lem}{\label{lemma 3}}
Let $V = (V_1, \ldots, V_n)$ be a twisted isometry. Then:
\begin{enumerate}
\item $V_j \clw_A = \clw_A$ for all $A \subsetneqq I_n$ and $j \in I_n\setminus A$.
\item $U_{st} \clw_A = \clw_A$ for all $s < t$ and $A \subseteq I_n$.
\item $V_j \clh_{V, A} = \clh_{V, A}$ for all $A \subsetneqq I_n$ and $j \in I_n\setminus A$.
\end{enumerate}
\end{lem}
\begin{proof}
Fix $A \subsetneqq I_n$ and $j \in I_n \setminus A$. Clearly, $\clw_A \subseteq V_{j} \clw_A$. For the nontrivial inclusion, observe that by Lemma \ref{lemma: twist commute}, there exists a monomial $\eta_1 \in \C[z_1, \ldots, z_{n-|A|}]$ such that
\[
\begin{split}
V_{j}  \Big(\bigcap_{l\in \Z_+^{n-|A|}} V^l_{I_n\setminus A}\mathcal{E}_A\Big) & = \bigcap_{l\in \Z_+^{n-|A|}} V^l_{I_n\setminus A} V_{j} \eta_1 (U) \cle_A
\\
& = \bigcap_{l\in \Z_+^{n-|A|}} V^l_{I_n\setminus A}V_{j} \cle_A
\\
& \subseteq \bigcap_{l\in \Z_+^{n-|A|}} V^l_{I_n\setminus A} \cle_A.
\end{split}
\]
where the last equality and the inclusion follow from Lemma \ref{lemma 1} and Lemma \ref{lemma:Ea invariant}, respectively. Hence $ V_{j} \clw_A = \clw_A$, which completes the proof of part (1).

\noindent Part (2) follows from the Lemma \ref{lemma 1} along with the fact that $U_{ij}$ commutes with $V_k$ for all $i \neq j$ and $k$.

\noindent Now we prove that $V_j \clh_{V, A} = \clh_{V, A}$, for all $A\subsetneqq I_n$ and $j \in A^c$. Fix $A\subsetneqq I_n$ and $j \in A^c$. Note that for each $k \in \Z_+^{|A|}$, there exists a monomial $\eta_k$ such that $V_j V_A^k = V_A^k V_j \eta_k(U)$. In view of $\eta_k(U) \clw_A = \clw_A$, we compute
\[
\begin{split}
V_j \clh_{V, A} & = V_j \Big(\underset{k\in\Z_+^{|A|}}{\bigoplus} V_A^k \clw_A\Big)
\\
& = \Big(\underset{k\in\Z_+^{|A|}}{\bigoplus} V_A^k V_j \eta_k(U)\clw_A\Big)
\\
& = \Big(\underset{k\in\Z_+^{|A|}}{\bigoplus} V_A^k V_j \clw_A\Big).
\end{split}
\]
Since $V_j \clw_A = \clw_A$, by part (1), it follows that $V_j \clh_{V, A} = \clh_{V, A}$, which completes the proof of the lemma.  	
\end{proof}

Wandering subspaces will play a key role in the remaining part of the paper.

\newsection{Twisted weak shifts}\label{sec: weak shift}
	
In this section, we will introduce the notion of twisted weak shifts. This will appear to be the right generalization of Popovici's weak bi-shifts for pairs of commuting isometries. We begin with the definition of twisted shifts which was introduced in \cite{RSS}.

\begin{definition}{\label{def: n shift}}
A twisted shift is an $n$-tuple of doubly twisted isometry $(V_1, \ldots, V_n)$ such that $V_i$ is a shift for all $i=1, \ldots, n$.
\end{definition}

Suppose $V = (V_1, \ldots, V_n)$ is a doubly twisted isometry on $\clh$. In view of Theorem \ref{thm: Wold for Un twisted}, $V$ admits a von Neumann-Wold decomposition $\clh = \bigoplus_{A \subseteq I_n} \clh_A$, where
\[
\clh_A = \bigoplus_{k \in \Z_+^{|A|}} V_A^k \Big(\bigcap_{l \in \Z_+^{n-|A|}}V^{l}_{I_n \setminus A} \cln_A\Big) \qquad (A \subseteq I_n).
\]
Therefore, $V$ is a twisted shift if and only if $\clh_A = \{0\}$ for all $A \neq I_n$. Equivalently, $\clh$ admits the following decomposition
\[
\clh = \bigoplus_{k \in \Z^n_+}V^k(\cln),
\]
where $\cln =: \cln_{I_n} = \bigcap_{i \in I_n} \ker V_i^*$.

The typical example of twisted shifts is built up from commuting unitary operators and shifts on $H^2(\D^n)$, the Hardy space over the open unit polydisc $\D^n$. Recall that $H^2(\D^n)$ is the Hilbert space of all square summable analytic functions on $\D^n$. Given a Hilbert space $\cle$, we denote by $H^2_{\cle}(\D^n)$ the $\cle$-valued Hardy space over $\D^n$. Then $(M_{z_1}, \ldots, M_{z_n})$ defines a doubly commuting shifts on $H^2_{\cle}(\D^n)$, where $M_{z_i} f = z_i f$ for all $f \in H^2_{\cle}(\D^n)$. It is often convenient to identify $H^2_{\cle}(\D^n)$ with $H^2(\D^n) \otimes \cle$. We need a definition:

\begin{definition}
Let $\cle$ be a Hilbert space, $U \in \clb(\cle)$ be a unitary, and let $j \in \{1, \ldots, n\}$. The \textit{$j$-th diagonal operator with symbol $U$} is the unitary operator $D_j[U]$ on $H^2_{\cle}(\D^n)$ defined by
\[
D_j[U] (z^k \eta) = z^k (U^{k_j} \eta) \qquad (k \in \Z_+^n, \eta \in \cle).
\]
\end{definition}

Let $\cle$ be a Hilbert space, and let $\{U_{ij}\}_{i<j}$ be a twist on $\cle$. Then $\{I_{H^2(\D^n)} \otimes U_{ij}\}_{i < j}$ defines a twist on $H^2_{\cle}(\D^n)$ (or on $H^2(\D^n) \otimes \cle$, to be more specific). Define
\[
V_i = \begin{cases}
M_{z_1} & \mbox{if } i=1
\\
M_{z_i} \Big(D_1 [U_{i 1}] D_2 [U_{i 2}] \cdots D_{i-1}[U_{i i-1}]\Big) & \mbox{otherwise}.
\end{cases}
\]
A routine computation (cf. \cite{RSS}) then reveals that $(V_1, \ldots, V_n)$ is a twisted shift on $H^2_{\cle}(\D^n)$. This is essentially a model example of twisted shifts. Observe that, since a twist is made by a family of commuting unitaries and a twisted shift is made of a twist, it is immediate that the class of doubly twisted isometries is larger than the doubly non-commuting isometries \cite{JP}.

In order to formulate the notion of twisted weak shifts, we need a characterization of twisted shifts. Suppose $V = (V_1, \ldots, V_n)$ is a twisted isometry. Clearly, $\Pi_{i \in I_n} V_i$ is an isometry. The following lemma, in particular, explains the unitary part of $\Pi_{i \in I_n} V_i$ in terms of $\clh_{\emptyset}$ as in Definition \ref{def: wand sub}.

\begin{lem}\label{lemma: intersection}
Let $V = (V_1, \ldots, V_n)$ be a twisted isometry on $\clh$, and let $\cls \subseteq \clh$ reduces $V$. Then
\[
\underset{m \in \Z_+}{\bigcap} (V_1 \cdots V_2)^m \cls = \underset{k \in \Z_+^n}{\bigcap}
V^k \cls.
\]
\end{lem}
\begin{proof}
We prove it only for $n=2$ as the remaining part can easily be proven by induction for any positive integer $n \geq 2$. Suppose $V = (V_1, V_2)$. Evidently $\bigcap_{m \in \Z_+}(V_1 V_2)^m \cls \supseteq \bigcap_{k_1, k_2 \in \Z_+} V_1^{k_1} V_2^{k_2} \cls$. For the reverse inclusion, suppose $x \in \bigcap_{m \in \Z_+}(V_1 V_2)^m \cls$. Let $U$ be the corresponding twist for $(V_1, V_2)$. Since $\cls$ reduce $V_1$ and $V_2$, it follows that $\cls$ also reduces $U$ (see \eqref{eqn: U = VV}). Let $k_1, k_2 \in \Z_+$ and suppose $k_1 < k_2$. There exists $h \in \cls$ (depending on $k_2$) such that $x = (V_1V_2)^{k_2} h$. By  Lemma \ref{lemma: twist commute} again, there exist monomials $\eta_1$ and $\eta_2$ such that
\[
\begin{split}
x & = (V_1V_2)^{k_1} (V_1V_2)^{k_2-k_1} h
\\
& = V_1^{k_1}V_2^{k_1}(V_1V_2)^{k_2-k_1} (\eta_1 (U^*) h)
\\
& = V_1^{k_1}V_2^{k_2} (V_1^{k_2-k_1} \eta_2 (U) \eta_1 (U^*)h).
\end{split}
\]
Since $\eta_2 (U) \eta_1 (U^*) \cls \subseteq \cls$, we have $x \in V^{k_1}_1 V^{k_2}_2 \cls$, which proves the reverse inclusion.
\end{proof}

We are now ready for the characterization of twisted shifts.

\begin{prop}\label{prop: n-shift}
Let $V = (V_1, \ldots, V_n)$ be a doubly twisted isometry on $\clh$. Then $V$ is a twisted shift if and only if $V_i|_{\cln_{\{i\}^c}}$ and $V_{j}V_{k}$ are shifts for all $i, j, k \in I_n$ and $j \neq k$.
\end{prop}
\begin{proof}
If $V$ is a twisted shift, then, by the fact that $\cln_{\{i\}^c}$ reduces $V_i$ (see Lemma \ref{lem: d twist basic lemma}), it follows that $V_i|_{\cln_{\{i\}^c}}$ is a shift. Moreover, by Lemma \ref{lemma: intersection}, we have
\[
\bigcap_{m \in \Z_+} (V_j V_k)^m \clh = \bigcap_{m_j, m_k \in \Z_+} (V_j^{m_j} V_k^{m_k}) \clh \subseteq \bigcap_{m_j \in \Z_+} V_j^{m_j} \clh = \{0\},
\]
as $V_j$ is a shift. Therefore, $V_{j}V_{k}$ is a shift for all $j \neq k$. For the converse, suppose $V_i|_{\cln_{\{i\}^c}}$ and $V_{j}V_{k}$ are shifts for all $j \neq k$ and $i \in I_n$. Our goal is to prove that $V_i$ is a shift, that is,  $\underset{k_i\geq 0}{\bigcap} V_{i}^{k_i}\clh=0$ for all $i \in I_n$. To this end, fix $i \in I_n$. Note that
\[
\tilde{V} = (V_1, \ldots, V_{i-1}, V_{i+1}, \ldots, V_n),
\]
is an $(n-1)$-tuple of doubly twisted isometries with respect to the twist
\[
\tilde U = \{U_{pq}: p < q \text{ and } p,q \neq i\}.
\]
Therefore, by Theorem \ref{thm: Wold for Un twisted}, $\tilde V$ admits von Neumann-Wold decomposition $\clh = \bigoplus_{A \subseteq \{i\}^c} \clh_A$, where
\[
\clh_A = \underset{k \in \Z_+^{|A|}} {\bigoplus} V^k_A \Big(\bigcap_{l \in \Z_+^{n-1-|A|}} V^l_{J \setminus A}(\mathcal{N}_A)\Big) \qquad (A \subseteq \{i\}^c),
\]
and $J = I_n \setminus \{i\}$. Fix $A \subseteq \{i\}^c$ and $k_i \in \Z_+$. Suppose $m = n-1-|A|$. By Lemma \ref{lemma: twist commute}, for each $k \in \Z_+^{|A|}$, there exists a monomial $\eta_k$ such that $V_i^{k_i} V_A^k = V_A^k V_i^{k_i} \eta_k(U)$. We compute
\[
\begin{split}
V_i^{k_i} \clh_A & = V_i^{k_i} \Big(\underset{k \in \Z_+^{|A|}}{\bigoplus} V^k_A \Big(\bigcap_{l \in \Z_+^m} V^l_{J \setminus A} (\mathcal{N}_A) \Big)\Big)
\\
& = \bigoplus_{k \in \Z_+^{|A|}} V^k_A V_i^{k_i} \eta_k(U) \Big(\bigcap_{l \in \Z_+^m} V^l_{J \setminus A} (\mathcal{N}_A)\Big)
\\
&= \bigoplus_{k \in \Z_+^{|A|}} V^k_A V_i^{k_i} \Big(\bigcap_{l \in \Z_+^m} V^l_{J \setminus A} (\mathcal{N}_A)\Big)
\\
&= \bigoplus_{k \in \Z_+^{|A|}} V^k_A \Big(\bigcap_{l \in \Z_+^{m}} V_i^{k_i} V^l_{J \setminus A} (\mathcal{N}_A)\Big),
\end{split}
\]
where the last but one equality follows from the fact that
\[
\eta_k(U) \Big(\bigcap_{l \in \Z_+^m} V^l_{J \setminus A} (\mathcal{N}_A)\Big) = \Big(\bigcap_{l \in \Z_+^m} V^l_{J \setminus A} (\mathcal{N}_A)\Big).
\]
Therefore
\[
\bigcap_{k_i \in \Z_+} V_i^{k_i} \clh_A = \bigoplus_{k \in \Z_+^{|A|}} V^k_A \Big(\bigcap_{k_i \in \Z_+, l \in \Z_+^m} V_i^{k_i} V^l_{J \setminus A}
(\mathcal{N}_A)\Big).
\]
If $A \varsubsetneqq J = \{i\}^c$, there exists $j \in \{i\}^c \setminus A$, such that
\begin{equation}\label{eqn: proof shift}
\underset{k_i \in \Z_+, l \in \Z_+^m}{\bigcap} V_i^{k_i} V^l_{J \setminus A} (\mathcal{N}_A) = \underset{k_i,{ l_j \in \Z_+, l' \in \Z_+^{m-1}}} {\bigcap} V_i^{k_i} {V_j^{l_j} \Big(V^{l'}_{J \setminus A \cup \{j\}} (\mathcal{N}_A)\Big)},
\end{equation}
where $l=(l_1,\cdots,l_m)\in \Z^m_+$. Applying Lemma \ref{lemma: intersection}, we have
\[
\bigcap_{k_i, k_j \in \Z_+} V^{k_i}_i  V^{l_j}_j \Big( V^{l'}_{J \setminus A \cup \{j\}} \mathcal {N}_A \Big) \subseteq \bigcap_{k_i, k_j \in \Z_+} V^{k_i}_i  V^{l_j}_j \clh = \bigcap_{m \in \Z_+} (V_i  V_j) ^m \clh = \{0\},
\]
for all $l' \in \Z_+^{m-1}$. Then \eqref{eqn: proof shift} implies
\[
\bigcap_{k_i \in \Z_+, l \in \Z_+^m} V_i^{k_i} V^l_{J \setminus A}
(\mathcal{N}_A) = \{0\}.
\]
Then the equality preceding \eqref{eqn: proof shift} yields $\bigcap_{i \in \Z_+} V_i^{k_i} \clh_A = \{0\}$ for all $A \varsubsetneqq \{i\}^c$. Therefore
\[
\begin{split}
\bigcap_{k_i \in \Z_+} V_i^{k_i} \clh & = \bigcap_{k_i \in \Z_+} V_i^{k_i} \clh_{\{i\}^c}
\\
& = \bigoplus_{k \in \Z_+^{n-1}} V^k_{\{i\}^c}\Big(\bigcap_{k_i\in \Z_+} V_i^{k_i}(\mathcal{N}_{\{i\}^c})\Big).
\end{split}
\]
Since $V_i|_{\mathcal{N}_{\{i\}^c}}$ is a shift by assumption, we obtain $\bigcap_{k_i\in \Z_+} V_i^{k_i} \clh=0$, which completes the proof of the proposition.
\end{proof}

Now we turn to twisted isometries. Let $V = (V_1, \ldots, V_n)$ be a twisted isometry. Recall from Definition \ref{def: weak wandering} the $A$-weak wandering subspace for $V$ is given by
\[
\cle_A = \bigcap_{i\in A}\underset{\substack{B \subseteq \{i\}^c}}{\bigcap} \underset{k\in \Z_+^{|B|}}{\bigcap} \ker V_i^*V_B^k,
\]
for all nonempty $A \subseteq I_n$, and $\cle_\emptyset = \clh$. Also recall from Definition \ref{def: wand sub} that the $A$-wandering subspace for $V$ is given by
\[
\clw_A = \underset{l\in \Z_+^{n-|A|}}{\bigcap}V^l_{I_n\setminus A}\mathcal{E}_A \qquad (A \subseteq I_n).
\]
Assume for a moment that $V$ is doubly twisted. Fix $i \in A$. Let $B \subseteq \{i\}^c$ and $k \in \Z_+^{|B|}$. Since $i \notin B$ and $V$ is doubly twisted, by Lemma \ref{lemma: twist commute}, there exists a monomial $\eta$ such that
\[
V_i^* V_B^k = (\eta(U) V_B^k) V^*_i.
\]
Since $\eta(U) V_B^k$ is an isometry, it follows that $\cle_A = \bigcap_{i\in A} \ker V_i^* = \cln_A$, and hence
\begin{equation}\label{eqn: W_A repr fod dc}
\clw_A = \underset{l\in \Z_+^{n-|A|}}{\bigcap}V^l_{I_n\setminus A}\mathcal{N}_A \qquad (A \subseteq I_n),
\end{equation}
the wandering subspace of doubly twisted isometries (see Theorem \ref{thm: Wold for Un twisted}). In view of this observation and Proposition \ref{prop: n-shift}, we are now in a position to define a weaker version of shift that fits appropriately in orthogonal decompositions of twisted isometries.

\begin{definition}[Twisted weak shift]\label{def: weak shift}
A twisted isometry $V=(V_1,\ldots, V_n)$ is said to be a twisted weak shift if
\begin{enumerate}
\item $V_i|_{\cle_{\{i\}^c}}$ is a shift for all $i \in I_n$, and
\item $V_jV_k|_{\cle_{\{j,k\}^c}}$ is a shift for all $j,k \in I_n$ and $j \neq k$.
\end{enumerate}
\end{definition}

Lemma \ref{lemma:Ea invariant} ensures that the above definition is consistent. Moreover, in the case of pairs of commuting isometries, the above definition coincides with Popovici's weak bi-shift.

\section {Twisted isometries}\label{sec: Popovici-Wold type decomposition}

In this section, we prove that a twisted isometry $V=(V_1, \ldots, V_n)$ admits orthogonal decompositions in the sense of Definition \ref{def: orth decomp}. Moreover, we prove that $V|_{\clh_{I_n}}$ is a twisted weak shift. For the case of commuting pairs of isometries, our result recovers the Popovici decomposition. First, we prove that $V$ indeed admits an orthogonal decomposition. Recall that the $A$-wandering subspace for $V$ is defined by (see Definition \ref{def: wand sub})
\[
\clw_A = \underset{l\in \Z_+^{n-|A|}}{\bigcap}V^l_{I_n\setminus A}\mathcal{E}_A \qquad (A \subseteq I_n),
\]
where (see Definition \ref{def: weak wandering}) $\cle_{\emptyset}=\clh$ and
\[
\mathcal{E}_A:=\bigcap_{i\in A}	\bigcap_{B \subseteq \{i\}^c}
\bigcap_{k\in \Z_+^{|B|}} \Big(\ker V_i^*V_B^k\Big),
\]
for all $A \subseteq I_n$ such that $A \neq \emptyset$.

\begin{prop}{\label{prop: doubly-commuting-part}}
Let $V=(V_1,\dots ,V_n)$ be a twisted isometry. For each $A \subsetneqq I_n$, define
\[
\clh_{V, A}:= \underset{k\in \Z_+^{|A|}}{\bigoplus}V_A^k \clw_A.
\]
Then the following holds:
\begin{enumerate}
\item $\clh_{V, A}$ reduces $V_i$ for all $i \in I_n$.
\item $V_i|_{\clh_{V,A}}$ is a shift for all $i\in A$.
\item $V_i|_{\clh_{V,A}}$ is a unitary for all $i\notin A$.
\item The $n$-tuple $V|_{\clh_{V, A}}$ is doubly twisted.
\end{enumerate}
Moreover, $\clh_{V, A}$ is maximal, that is, if a closed subspace $\clk_{V, A} \subseteq \clh$ satisfies the above four conditions, then $\clk_{V, A} \subseteq \clh_{V, A}$.
\end{prop}
\begin{proof}
Suppose $A=\emptyset$. Then $\clh_{V,\emptyset} = \bigcap_{k \in \Z_+^n} V^k \clh$. By part (3) of Lemma \ref{lemma 3}, for each $i \in I_n$, we have $V_i \clh_{V,\emptyset} = \clh_{V,\emptyset}$ and hence $V_i^* \clh_\emptyset = \clh_\emptyset$. Therefore, $\clh_{V,\emptyset}$ reduces $V_i$ and $V_i|_{\clh_{V,\emptyset}}$ is a unitary for all $i \in I_n$. Next, suppose $A (\neq \emptyset)$ is a proper subset of $I_n$. Fix $i \in A$. For each $k \in \Z_+^{|A|}$, there exists a monomial $\eta_k$ such that $V_i  V^k_A = V^k_A V_i \eta_k(U)$ (see Lemma \ref{lemma: twist commute}). By part (2) of Lemma \ref{lemma 3}, we know that $\eta_k(U) \clw_A = \clw_A$. Therefore
\[
\begin{split}
V_i \clh_{V,A} & = V_i\Big(\bigoplus_{k \in \Z_+^{|A|}} V^k_A \clw_A\Big)
\\
& = \bigoplus_{k \in \Z_+^{|A|}} V^k_A V_i \Big(\eta_k(U) \clw_A\Big)
\\
& = \bigoplus_{k \in \Z_+^{|A|}} V^{k+e_i}_A \clw_A,
\end{split}
\]
and hence $V_i \clh_{V,A} \subseteq \clh_{V,A}$. Now we prove that $V_i^* \clh_{V,A} \subseteq \clh_{V,A}$. Fix $k \in \Z_+^{|A|}$. If $k_i=0$, then, $V_i^*V_A^k(\xi)=0$, as $\xi\in \clw_A \subseteq \cle_A$. If $k_i>0$, then by Lemma \ref{lemma: twist commute}, there exist monomials $\eta_1$ and $\eta_2$ such that
\begin{align*}
V_i^*V_A^{k}(\xi)
&= V_{i}^* V^{k_i}_i V_{A}^{k-k_ie_i} \eta_1 (U^*)(\xi)\\
&= V_{i}^{k_i-1}V_{A}^{k-k_ie_i}\eta_1(U^*)(\xi)
= V_A^{k-e_i}\eta_2(U)\eta_1 (U^*)(\xi).
\end{align*}
Again, by part (2) of Lemma \ref{lemma 3}, $\eta_2(U)\eta_1 (U^*)(\xi) \in \clw_{A}$. This implies $\clh_{V, A}$ reduces $V_i$ for all $i\in A$. If $i\notin A$, then part (3) of Lemma \ref{lemma 3} yields $V_i\clh_{V, A}= \clh_{V, A}$, which completes the proof of (1).

\noindent Next, for each $i \in I_n$, we set $\tilde{V_i}:=V_i|_{\clh_{V, A}}$ and $\tilde V:=V|_{\clh_{V, A}}$. Clearly, $\tilde{V_i}$ is a shift for all $i\in A$ and unitary for all $i\notin A$. This proves (2) and (3). To check (4), we set
\[
\tilde U_{ij}:=U_{ij}|_{\clh_{V, A}} \qquad (i \neq j).
\]
Clearly, $\tilde{V_i}\tilde{V_j} = \tilde U_{ij} \tilde{V_j}\tilde{V_i}$ for all $i \neq j$. Fix $i \neq j$ in $I_n$. If any one of $\tilde{V_i}$ and $\tilde{V_j}$ is unitary, then evidently $\tilde{V_i}^*\tilde{V_j}=\tilde{U_{ij}}^*\tilde{V_j}\tilde{V_i}^*$. Suppose now that both $\tilde{V_i}$ and $\tilde{V_j}$ are shifts. Of course, in this case $i,j\in A$. Let $V_A^{k}(\xi)\in \clh_{V, A}$ for some $k \in \Z^{|A|}_+$ and $\xi\in \clw_A$. If $k_i=0$, then as $\xi\in \clw_A \subseteq \mathcal{E}_A$, it follows that
\[
\tilde V_i^*\tilde V_j	\left(V_{A}^{k}\xi\right) = 0	= \tilde V_j\tilde V_i^*\left(V_{A }^{k}\xi\right).
\]
If $k_i>0$ and $i\leq j$, then by Lemma \ref{lemma: twist commute}, there exist monomials $\eta _1, \eta _2$, and $\eta _3$ such that
\begin{align*}
\tilde V_i^*\tilde V_j \left(V_A^{k}\xi\right)
& = V_i^* V_A^{k+e_j}\eta _1 (U)\xi
\\
& = V_i^*V_i^{k_i} V_{A}^{k - k_i e_i + e_j}\eta _{2} (U^*)\eta _1 (U)\xi
\\
&= V_A^{k-e_i+e_j}\eta _3 (U)\eta _{2} (U^*)\eta _1 (U)\xi,
\end{align*}
and
\begin{align*}
\tilde V_j \tilde V_i^*\left(V_A^{k}\xi\right) & = V_jV_i^*V_A^{k}\xi
\\
&= V_jV_i^*V_i^{k_i}V_{A }^{k - k_i e_i}\eta _{2} (U^*)\xi
\\
& = V_j V_A^{k-e_i}\eta _3 (U)\eta _{2} (U^*)\xi
\\
& = V_A^{k-e_i+e_j}(U_{ji}^*)\eta _1 (U)\eta _3 (U)\eta _{2} (U^*)\xi
\\
& = U_{ji}^*V_i^*V_jV_A^{k}\xi
\\
& = \tilde U_{ji}^*\tilde V_i^*\tilde V_j\left(V_A^{k}(\xi)\right),
\end{align*}
where the last but one equality follows from the fact that $V_j	V_A^{k} = V_A^{k+e_j}\eta _1(U)$, and
\[
V_j V_A^{k-e_i} = V_A^{k-e_i+e_j}(U_{ji}^*)\eta _1(U),
\]
whenever $i<j$. Similarly, for $i>j$, it follows that $\tilde{V}_i^*\tilde{V}_j = \tilde{U}_{ij}^* \tilde{V}_j \tilde{V}_i^*$, which proves that $V|_{\clh_{V,A}}$ is doubly twisted corresponding to the twist $\{\tilde U_{ij}\}_{i < j}$.

\noindent Now we prove the maximality of $\clh_{V,A}$. Consider a closed subspace $\clk_{V, A} \subseteq \clh$ satisfying all the four conditions. Applying Theorem \ref{thm: Wold for Un twisted} to the doubly twisted isometry $V|_{\clk_{V, A}}$, we have
\[
\clk_{V, A} = \bigoplus_{k\in \Z_+^{|B|}}
V_A^k \Big(\underset{l\in \Z_+^{n-|A|}}{\bigcap}V^l_{I_n\setminus A} \Big(\underset{i\in A}{\bigcap} \ker V_i^*\bigcap \clk_{V, A}\Big)\Big).
\]
Let $i \in A$, $B\subseteq \{i\}^c$, and suppose $x\in (\cap_{i\in A} \ker V_i^*) \cap \clk_{V, A}$. Since $V|_{\clk_{V, A}}$ is doubly twisted, it follows that $V_i^*V_B^kx=0$, and hence $x \in \mathcal{E}_A$. Therefore, if $h\in \clk_{V, A}$, then for any $l \in \Z_+^{|B|}$ and $B \subseteq A^c$, we can write
\[
h=\sum_{k\in \Z_+^{|B|}} V_A^k(V^l_B x_B),
\]
for some $x_B \in (\cap_{i\in A} \ker V_i^*) \cap \clk_{V, A}$. Since $x_B \in \mathcal{E}_A$, it follows that $h\in \clh_{V,A}$, which proves that $\clh_{V,A}$ is maximal and completes the proof of the proposition.
\end{proof}

We are finally ready for orthogonal decompositions of twisted isometries.

\begin{thm}\label{thm: twisted dec final}
Let $V=(V_1,\dots ,V_n)$ be a twisted isometry on $\clh$. Then there is a unique orthogonal decomposition	
\[
\clh = \bigoplus_{A \subseteq I_n}\clh_{V,A},
\]
where
\begin{enumerate}
\item $\clh_{V,A}$ reduces $V$ for all $A \subseteq I_n$.
\item $\clh_{V,A}$ is maximal for all $A \subsetneqq I_n$ in the sense that $V_i|_{\clh_{V,A}}$ is a shift if $i \in A$ and $V_j|_{\clh_{V,A}}$ is a unitary if $j \in A^c$.
\item $V|_{\clh_{V,I_n}}$ is a twisted weak shift.
	\end{enumerate}
\end{thm}
\begin{proof}
For $A\subsetneqq I_n$, define $\clh_{V,A}$ as in Proposition \ref{prop: doubly-commuting-part} and set $\clh_{I_n}:=\clh \ominus \bigoplus_{A \subsetneqq I_n }\clh_{V,A}$. Then (1) and (2) follow from Proposition \ref{prop: doubly-commuting-part}. For the last part we need to prove that
\begin{enumerate}
	\item[(a)] $V_i|_{\mathcal{E}_{\{i\}^c}\cap \clh_{V, I_n}}$ is a shift for all $i \in I_n$, and
	\item[(b)] $V_j V_k|_{\mathcal{E}_{\{j,k\}^c} \cap \clh_{V, I_n}}$ is a shift for all $j\neq k$.
\end{enumerate}
To this end, fix $i,j,k \in I_n$. Let $\{i\} = I_n\setminus A$ and
$\{j,k\} = I_n\setminus B$. Then by the definition of $\clh_{V, A}$ and $\clh_{V, B}$ (see Proposition \ref{prop: doubly-commuting-part}), we have
\[
\clh_{V,A}= \underset{k\in \Z_+^{|A|}}{\bigoplus}
V_A^k\Big(\underset{m\in \Z_+}{\bigcap}
V^m_i\mathcal{E}_{\{i\}^c}\Big), \text{ and } \clh_{V,B}= \underset{k\in \Z_+^{|B|}}{\bigoplus} V_B^k \Big(\underset{l\in \Z_+^2}{\bigcap}
V^l_{\{j,k\}} \mathcal{E}_{\{j,k\}^c} \Big).
\]
Since $\clh_{V, I_n} \perp \clh_{V,C}$ for all $C \subsetneq I_n$, in particular, we have
\[
\clh_{V, I_n} \bigcap \Big(\underset{m\in \Z_+}{\bigcap}
V^m_i\mathcal{E}_{\{i\}^c}\Big) = \{0\}
= \clh_{V, I_n} \bigcap \Big(\underset{l\in \Z_+^2}{\bigcap}
V^l_{\{j,k\}}\mathcal{E}_{\{j,k\}^c}\Big),
\]
that is, the unitary parts of $V_i|_{\mathcal{E}_{\{i\}^c}\cap \clh_{V, I_n}}$ and $V_j V_k|_{\mathcal{E}_{\{j,k\}^c} \cap \clh_{V, I_n}}$ are trivial. This proves that $\clh_{V, I_n}$ is a twisted weak shift.
\end{proof}

Clearly, in the case of commuting pairs of isometries, our result recovers the Popovici decomposition. It is also worth noting that the extension of Popovici decomposition from pairs of isometries to $n$-tuples of commuting isometries, $n > 3$, appears to be somewhat less obvious. For instance, the choice of wandering subspaces and the weak shift counterpart for $n$-tuples of isometries does not directly follow from the case of pairs of isometries.

\bigskip

\noindent\textbf{Acknowledgement:} The first author thanks Harish-Chandra Research Institute, Prayagraj, and ISI Bangalore for the Postdoctoral fellowship. The second named author is supported in part by Core Research Grant (CRG/2019/000908), by SERB (DST), Government of India.

\end{document}